\documentclass[a4paper,11pt]{article}
\usepackage[utf8]{inputenc}
\usepackage[T1]{fontenc}

\usepackage{amsthm}
\usepackage{tikz}
\usepackage{mathrsfs} 
\usepackage{fullpage}
\usepackage{thmtools}
\usepackage{thm-restate}
\usepackage{mathtools}
\usepackage{amssymb}
\usepackage{lmodern, comment}
\usepackage{graphicx}
\usepackage{pgf,tikz,tkz-graph,subcaption}
\usetikzlibrary{arrows,shapes}
\usetikzlibrary{decorations.pathreplacing}
\usepackage{tkz-berge}
\usepackage{enumerate}
\usepackage[normalem]{ulem}
\usepackage{url}
\usepackage{hyperref}
\usepackage[hyphenbreaks]{breakurl}
\usepackage{cite}
\usepackage{hyperref,comment}
\usepackage[capitalise]{cleveref}
\hypersetup{colorlinks = true, linkcolor = blue, citecolor = blue, urlcolor = blue}

\allowdisplaybreaks

\usepackage[margin=1in]{geometry}
\usepackage[textsize=small]{todonotes}

\usepackage{bookmark}

\newtheorem{problem}{Problem}
\newtheorem{theorem}{Theorem}
\newtheorem*{theorem*}{Theorem}

\newtheorem{lemma}[theorem]{Lemma}

\newtheorem{proposition}[theorem]{Proposition}
\newtheorem{definition}[theorem]{Definition}

\newtheorem{claim}[theorem]{Claim}
\newcommand*{\myproofname}{Proof}
\newenvironment{claimproof}[1][\myproofname]{\begin{proof}[#1]}{\end{proof}}
\newtheorem*{claim*}{Claim}

\newtheorem*{question*}{Question}

\newcommand*{\ceilfrac}[2]{\mathopen{}\left\lceil\frac{#1}{#2}\right\rceil\mathclose{}}

\newcommand*{\abs}[1]{\lvert #1\rvert}
\newcommand\blfootnote[1]{%
 \begin{NoHyper}
 \renewcommand\thefootnote{}\footnote{#1}%
 \addtocounter{footnote}{-1}%
 \end{NoHyper}
}

\def\eps{\varepsilon}

\title{Bounding mean orders of sub-$k$-trees of $k$-trees}

\author{Stijn Cambie \and Bradley McCoy \and Stephan Wagner \and Corrine Yap}
\date{\today}

\AtEndDocument{%
 \small{\begin{tabular}{@{}l@{}}%
 \textsc{Department of Computer Science, KU Leuven Campus Kulak, Kortrijk, Belgium}\\
 \textit{Email address}: \texttt{stijn.cambie@hotmail.com}\\
 \\
 \textsc{School of Computing, Montana State University, Bozeman, MT, USA}\\
 \textit{Email address}: \texttt{bradleymccoy@montana.edu}\\
 \\
 \textsc{Department of Mathematics, Uppsala University, Uppsala, Sweden}\\
 \textit{Email address}:
 \texttt{stephan.wagner@math.uu.se}\\
 \\
 \textsc{School of Mathematics, Georgia Institute of Technology, Atlanta, GA, USA}\\
 \textit{Email address}:
 \texttt{math@corrineyap.com}
 \end{tabular}}}

\begin{document}
\parindent=0cm

\maketitle

\begin{abstract}
    For a $k$-tree $T$, we prove that the maximum local mean order is attained in a $k$-clique of degree $1$ and that it is not more than twice the global mean order. We also bound the global mean order if $T$ has no $k$-cliques of degree $2$ and prove that for large order, the $k$-star attains the minimum global mean order.
    These results solve the remaining problems of Stephens and Oellermann [J. Graph Theory 88 (2018), 61-79] concerning the mean order of sub-$k$-trees of $k$-trees.
\end{abstract}

\blfootnote{2020 Mathematics Subject Classification. Primary 05C05, 05C35}
\blfootnote{Keywords: trees, $k$-trees, average subtree order}

\vspace{-1cm}
\section{Introduction}\label{sec:intro}

In \cite{jamison_average_1983} and \cite{jamison_monotonicity_1984} Jamison considered the mean number of nodes in subtrees of a given tree. He showed that for trees of order $n$, the average number of nodes in a subtree of $T$ is at least $(n + 2)/3$, with this minimum achieved if and only if $T$ is a path. He also showed that the average number of nodes in a subtree containing a root is at least $(n + 1)/2$ and always exceeds the average over all unrooted subtrees. The mean subtree order in trees was further investigated, e.g. in~\cite{VW10,John14,CWW21,LXWW23,Ruoyu23}, as well as extensions to arbitrary graphs~\cite{CGMV18,CM21,CCHT23,CWL23} and the mean order of the connected induced subgraphs of a graph~\cite{Vince20,Vince21,John22,John22b,Vince22}.

In~\cite{SO18}, Stephens and Oellermann extended the study to $k$-trees and families of sub-$k$-trees. A \emph{$k$-tree} is a generalization of a tree that has the following recursive construction.

\begin{definition}[$k$-tree]
Let $k$ be a fixed positive integer.
 \begin{enumerate}
 \item The complete graph $K_k$ is a $k$-tree.
 \item If $T$ is a $k$-tree, then so is the graph obtained from $T$ by joining a new vertex to all vertices of some $k$-clique of $T$.
 \item There are no other $k$-trees.
 \end{enumerate}
\end{definition}

Note that for $k=1$ we have the standard recursive construction of trees. A \emph{sub-$k$-tree} of a $k$-tree $T$ is a subgraph that is itself a $k$-tree. Let $S(T)$ denote the collection of all sub-$k$-trees of $T$ and let $N(T):=|S(T)|$ be the number of sub-$k$-trees. We denote by $R(T)= \sum_{X \in S(T)} \abs X$ the total number of vertices in all sub-$k$-trees (where for a graph $G$ the notation $|G|$ is used throughout to mean the number of vertices in $G$). The \emph{global mean (sub-$k$-tree) order} is 
$$\mu(T)=\frac{R(T)}{N(T)}.$$

For an arbitrary $k$-clique $C$ of $T$, let $S(T;C)$ denote the collection of sub-$k$-trees containing $C$ and let $N(T;C):=|S(T;C)|$. The \emph{local clique number} is $R(T;C)= \sum_{X \in S(T;C)} \abs X$ and the \emph{local mean (sub-$k$-tree) order} is
$$\mu(T;C)=\frac{R(T;C)}{N(T;C)}.$$ The \emph{degree} of $C$ is the number of $(k+1)$-cliques that contain $C$.

Stephens and Oellermann concluded their study of the mean order of sub-$k$-trees of $k$-trees with several open questions. Our main contribution here is to answer three of them.

It was conjectured by Jamison in \cite{jamison_average_1983} and proven by Vince and Wang in \cite{VW10} that for trees of order $n$ without vertices of degree $2$---called {\em series-reduced} trees---the global mean subtree order is between $\frac{n}{2}$ and $\frac{3n}{4}$. For $k$-trees we provide similar asymptotically sharp bounds, answering \cite[Problem 6]{SO18}.

\begin{restatable}
{theorem}{seriesreduced}
\label{thr:prob6}
 For every $k$-tree $T$ without $k$-cliques of degree $2,$ the global mean sub-$k$-tree order satisfies
 $$ \frac {n+k}2 -o_n(1) < \mu(T)< \frac{3n+k-3}{4}.$$
 These bounds are asymptotically sharp. In particular, for large $k$, $\frac{3n}{4}$ is not an upper bound. 
 
For large $n$, the $k$-star is the unique extremal $k$-tree for the lower bound.
\end{restatable}

Wagner and Wang proved in \cite{WW16} that the maximum local mean subtree order occurs at a leaf or a vertex of degree $2$. We prove an analogous result for $k$-trees, answering \cite[Problem 4]{SO18}. In contrast to the result for trees, it turns out that for $k \ge 2$, the maximum can only occur at a $k$-clique of degree $1$, unless $T$
 is a $k$-tree of order $k+2$.

\begin{restatable}
{theorem}{maxlocalmean}
\label{thr:prob4}
Suppose that $k \ge 2$. 
For a $k$-tree $T$ of order $n \not=k+2$, if a $k$-clique $C$ maximizes $\mu(T; C)$, then $C$ must be a $k$-clique of degree $1$. For $n=k+2,$ every $k$-clique $C$ satisfies $\mu(T; C)=k+1.$
\end{restatable}

Lastly, Jamison \cite{jamison_average_1983} conjectured that for a given tree $T$ and any vertex $v$, the local mean order is at most twice the global mean order of all subtrees in $T$. Wagner and Wang \cite{WW16} proved that this is true. Answering \cite[Problem 2]{SO18} affirmatively, we show

\begin{restatable}
{theorem}{localboundedbyglobal}
\label{thr:prob2}
 The local mean order of the sub-$k$-trees containing a fixed $k$-clique $C$ is less than twice the global mean order of all sub-$k$-trees of $T$.
\end{restatable}

\subsection{Related Results} A total of six questions were posed in \cite{SO18}. Problems 1 and 3 were solved by Luo and Xu~\cite{LX23}. Regarding the first problem, Jamison \cite{jamison_average_1983} showed that for any tree $T$ and any vertex $v$ of $T$, the local mean order of subtrees containing $v$ is an upper bound on the global mean order of subtrees of $T$. Stephens and Oellermann asked about a generalization to $k$-trees, to which Luo and Xu showed:

\begin{theorem}[\hspace{1sp}\cite{LX23}]
 For any $k$-tree $T$ of order $n$ with a $k$-clique $C$, we have $\mu(T; C) \geq\mu(T)$ with equality if and only if $T\cong K_k$.
\end{theorem}

For the third problem, it was shown in \cite{jamison_average_1983} that paths have the smallest global mean subtree order. For $k$-trees we have:

\begin{theorem}[\hspace{1sp}\cite{LX23}]
For any $k$-tree $T$ of order $n$, we have $\mu(T)\ge \frac{\binom{n-k+2}{3}}{\binom{n-k+1}{2} + (n-k)k + 1 }+ k$ with equality if and only if $T$ is a path-type $k$-tree.
\end{theorem}

Very recently, Li, Ma, Dong, and Jin~\cite{LMDJ23} gave a partial proof of Theorem~\ref{thr:prob4}, showing that the maximum local mean
order always occurs at a $k$-clique of degree $1$ or $2$, thus  also solving~\cite[Problem 4]{SO18}.
They did this by combining the fact that the $1$-characteristic trees of adjacent $k$-cliques can be obtained from each other by a partial Kelmans operation \cite[Lem.~4.6]{LMDJ23}, an inequality between local orders of neighboring vertices after performing a partial Kelmans operation \cite[Thm.~3.3]{LMDJ23}, as well as \cite[Lem.~11]{SO18} and \cite[Thm.~3.2]{WW16}.
Theorem~\ref{thr:prob4} also solves Problem~5.4 from~\cite{LMDJ23}, asking whether the maximum local mean
order can ever occur at a $k$-clique of degree $2$, but
not at a $k$-clique of degree $1$.

{\bf Outline:} In Section~\ref{sec:notation}, we go over definitions and notation. Theorems~\ref{thr:prob6}, \ref{thr:prob4}, \ref{thr:prob2} are proven in Sections~\ref{sec:prob6}, \ref{sec:prob4}, \ref{sec:prob2} respectively.
We additionally address \cite[Problem 5]{SO18}, which is a more general question asking what one can say about the local mean order of sub-$k$-trees containing a fixed $r$-clique for $1 \leq r \leq k$. We give a possible direction and partial results in the concluding section.

\section{Notation and Definitions}\label{sec:notation}

The global mean sub-$k$-tree order $\mu(T)$ and the local mean sub-$k$-tree order $\mu(T;C)$ are defined in the introduction. 
The local mean order always counts the $k$ vertices from $C$ and it sometimes will be more convenient to work with the average number of additional vertices in a (uniform random) sub-$k$-tree of $T$ containing $C$, in which case we use the notation $\mu^\bullet(T; C)=\mu(T; C)-k.$ Moreover, the number of sub-$k$-trees not containing $C$ will be denoted by $\overline N(T; C) = N(T)-N(T; C) $ and the total number of vertices in the sub-$k$-trees that do not contain $C$ will be denoted by $\overline R(T; C) =R(T)-R(T; C)$.

A {\em $k$-leaf} or {\em simplicial vertex} is a vertex belonging to exactly one $(k+1)$-clique of $T$, i.e., a vertex of degree $k.$
A {\em simplicial $k$-clique} is a $k$-clique containing a $k$-leaf.
Note that a $k$-clique of degree $1$ is not necessarily simplicial. A {\em major $k$-clique} is a $k$-clique with degree at least $3$.
Two $k$-cliques are {\em adjacent} if they share a $(k-1)$-clique.

The {\em stem} of a $k$-tree $T$ is the $k$-tree obtained by deleting all $k$-leaves from $T$.

Subclasses of trees generalize to subclasses of $k$-trees. We will reference two in particular: paths generalize to {\em path-type $k$-trees} that are either isomorphic to $K_k$ or $K_{k+1}$ or have precisely two $k$-leaves. 
Note that for every $n \ge k+3$ and $k \geq 3$, there are multiple non-isomorphic path-type $k$-trees of order $n.$

A {\em $k$-star} is either $K_k$ or $K_{k+1}$, or it is the unique $k$-tree with $n-k$ simplicial vertices when $n\ge k+2.$

Furthermore, the combination of a $k$-star and a $k$-path is called a $k$-broom:
take a $k$-path of a certain length and add some simplicial vertices to a simplicial $k$-clique of the $k$-path.

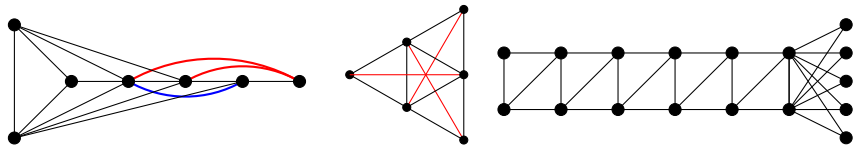
\begin{figure}[h]
    \centering
    \begin{tikzpicture}[scale=0.75]
        \draw (5,0) -- (1,0);
        \draw (0,1) -- (0,-1);
        \draw[thick,red] (5,0) arc (60:120:3cm);
        \draw[thick,red] (5,0) arc (60:120:2cm);
        \draw[thick,blue] (2,0) arc (240:300:2cm);
        
        \foreach \i in {1,2,3}
        {
            \draw (0,-1)--(\i,0) -- (0,1);
            \draw[fill=black] (\i,0) circle (3pt);
        }
         \draw (0,-1)--(4,0);

 \foreach \i/\j in {0/1,0/-1,5/0,4/0} {
        \draw[fill=black] (\i,\j) circle (3pt);
        }
    \end{tikzpicture}
    \quad 
    \begin{tikzpicture}[scale=0.5]

        \foreach \i in {120,240,0}
        {
        \draw (\i:1)--(\i+120:1);
        \draw (\i:1)--(\i+60:2);
        \draw (\i:1)--(\i-60:2);
        \draw[red] (\i:1)--(\i-180:2);
        \draw[fill=black] (\i:1) circle (3pt);
        \draw[fill=black] (\i+60:2) circle (3pt);
        }

    \end{tikzpicture}\quad 
    \begin{tikzpicture}[scale=0.75]
        \draw (0,0) -- (5,0)--(5,1)--(0,1)--(0,0);

        \foreach \i in {0,1,2,3,4}
        {
            \draw (\i,0)--(\i+1,1) -- (\i+1,0);
            \draw[fill=black] (\i,0) circle (3pt);
            \draw[fill=black] (\i,1) circle (3pt);
        }

        \foreach \y in {-0.5,0,0.5,1,1.5}
        {
            \draw (5,0)--(6,\y) -- (5,1);
            \draw[fill=black] (6,\y) circle (3pt);
        }

 \foreach \i/\j in {5/1,5/0} {
        \draw[fill=black] (\i,\j) circle (3pt);
        }
    
    \end{tikzpicture}
    \caption{Examples of a $3$-path, $3$-star and $2$-broom}
    \label{fig:2path_star_broom}
\end{figure}

For a given $k$-clique $C$ in $T$, it is often useful to decompose $T$ into $C$ and the sub-$k$-trees that result from deleting $C$.
Let $B_1,B_2,\ldots,B_d$ be the $(k+1)$-cliques that contain $C$, and let $v_i$ be the vertex of $B_i \setminus C$. Moreover, let $C_{i,1},\ldots,C_{i,k}$ be the $k$-subcliques of $B_i$ other than $C$. The $k$-tree $T$ can be \emph{decomposed} into $C$ and $k$-trees $T_{1,1},\ldots,T_{d,k}$, rooted at $C_{1,1},\ldots,C_{d,k}$ respectively, that are pairwise disjoint except for the vertices of the cliques $C_{i,j}$.

In a $1$-tree, any two vertices are connected by a unique path. 
This fact generalizes to $k$-trees through the construction of the 1-characteristic tree of a $k$-tree $T$~\cite{LX23,SO18}.
For a $k$-clique $C$ in $T$, a {\em perfect elimination ordering of $T$ to $C$} is an ordering $v_1, v_2, \ldots, v_{n-k}$ of its vertices different from $V(C)=\{c_1, \ldots, c_k\}$ such that each vertex $v_i$ is simplicial in the $k$-tree spanned by $C$ and $v_j, 1 \le j \le i$.
In~\cite{SO18} it is shown that for any $v \notin C$, there is a unique sequence of vertices that along with $C$ induce a path-type $k$-tree $P(C, v)$ and that form a perfect elimination ordering of $P(C, v)$ to $C$. It is also proven that $T$ can be written as $\bigcup P(C,v)$ where the union is taken over all $k$-leaves $v \in V$. Each $k$-tree $P(C,v)$ has an associated 1-tree $P'(C,v)$ where the vertices consist of a single vertex representing the entire clique $C$, along with the remaining non-$C$ vertices of $P(C,v)$. The edges are consecutive pairs from the perfect elimination ordering. Taking $\bigcup P'(C,v)$ over all $k$-leaves $v$ gives us what is called the {\em $1$-characteristic tree} of $T$, which we will denote $T'_C$. See Figure~\ref{fig:characteristic} for an example.

\begin{figure}[h]
    \centering
	\tikzstyle{every node}=[circle, draw, fill=black,
                        inner sep=0pt, scale=.6pt]
	\begin{tikzpicture}
   
	    	\node [label={[shift={(0,-1.25)}]\huge$c_1$}] (0) at (2,-.5)  {0};
 	   	\node [label={[shift={(0,0)}]\huge$c_2$}] (1) at (2,.75) {1};
	    	\node [label={[shift={(0,-1.25)}]\huge$v_1$}] (2) at (0,-.75) {2};
  		\node [label={[shift={(0,0)}]\huge$v_2$}] (3) at (0,.6) {3};
 		\node [label={[shift={(0,-1.25)}]\huge$v_3$}] (4) at (3.4,-.4) {4};
  		\node [label={[shift={(0,0)}]\huge$v_4$}] (5) at (3.5,.9) {5} ;
 		\node [label={[shift={(0,-1.25)}]\huge$v_5$}]  (6) at (5.25,-.75) {6} ;
 		\node [label={[shift={(0,0)}]\huge$v_6$}] (7) at (5.25,.75) {7} ;
  
    \path [line width=1.mm] (0) edge (1);
    \begin{scope}[
    every edge/.style={draw=black, thick}] 
   		\path [-] (1) edge (2);
   		\path [-] (0) edge (2);
   		\path [-] (0) edge (3);
   		\path [-] (1) edge (3);
    	\path [-] (4) edge (0);
   		\path [-] (4) edge (5);
  		\path [-] (4) edge (1);
  	 	\path [-] (5) edge (1);
	 	\path [-] (4) edge (7);
	 	\path [-] (4) edge (6);
	 	\path [-] (5) edge (6);
 		\path [-] (5) edge (7);
   	 \end{scope}
    \end{tikzpicture}\hspace{1cm}
     \begin{tikzpicture}
   
	    	\node [label={[shift={(0,0)}]\huge$C$}] (0) at (1.5,0)  {0};
	    	\node [label={[shift={(0,-1.25)}]\huge$v_1$}] (1) at (.25,-.4) {1};
  		\node [label={[shift={(0,0)}]\huge$v_2$}] (2) at (.25,.6) {2};
 		\node [label={[shift={(0,0)}]\huge$v_3$}] (3) at (3,-.075) {3};
  		\node [label={[shift={(0,-1.25)}]\huge$v_4$}] (4) at (4.7,.25) {4} ;
 		\node [label={[shift={(0,-1.25)}]\huge$v_5$}]  (5) at (6,-.75) {5} ;
 		\node [label={[shift={(0,0)}]\huge$v_6$}] (6) at (6,.6) {6} ;
  
    \begin{scope}[>={stealth[black]},
    every edge/.style={draw=black, very thick}] 
   		\path [-] (0) edge (2);
   		\path [-] (0) edge (3);
    	\path [-] (0) edge (1);
   		\path [-] (3) edge (4);
	 	\path [-] (4) edge (5);
 		\path [-] (4) edge (6);
   	 \end{scope}
    \end{tikzpicture}
     \caption{A $2$-tree (left) with $2$-clique $C = \{c_1,c_2\}$ and the $1$-characteristic tree (right).}
    \label{fig:characteristic}
\end{figure}
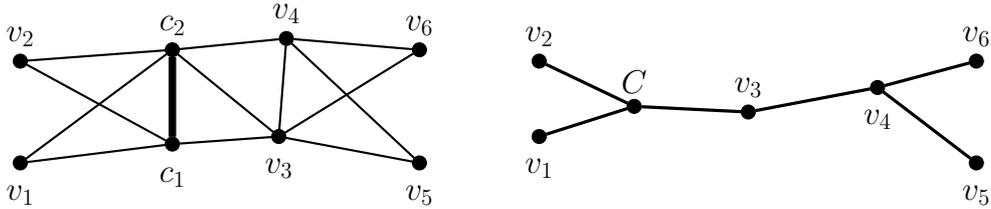

\section{Excluding $k$-cliques of degree $2$}\label{sec:prob6}

In this section, we prove 
\seriesreduced*

\subsection{The lower bound}

Among series-reduced trees of order $n$, for $4 \le n \le 8$ the star attains the maximum mean subtree order,
but for $n\ge 11$ it attains the minimum mean subtree order. 
It can be derived from~\cite[Lem.~12]{John14} and an adapted version of~\cite[Cor.~11]{John14} (with $2$ replaced by any $\eps>0$, at the cost of replacing $30$ by $n_{\eps}$) that this is indeed the case for $n$ sufficiently large.\footnote{We thank John Haslegrave for this remark}
In~\cite{C23}, it is shown that for $6 \le n \le 10$, the series-reduced trees attaining the minimum mean subtree order are those presented in Figure~\ref{fig:5nonstar_extremalgraphs}, and for $n \ge 11,$ $S_n$ is always the unique extremal graph.

\begin{figure}
    \centering

\begin{tikzpicture}[scale=0.5]
        \draw (4,1) -- (2,1);
        \foreach \i/\j in {1/0,1/2}
        {
 \draw (\i,\j) -- (2,1);
        }
        
        \foreach \i/\j in {5/2,5/0} {
 \draw (\i,\j) -- (4,1);
        }

 \foreach \i/\j in {1/0,1/2,5/2,5/0,4/1,2/1} {
        \draw[fill=black] (\i,\j) circle (3pt);
        }
    \end{tikzpicture}
    \quad 
    \begin{tikzpicture}[scale=0.5]
        \draw (4,1) -- (2,1);
        \foreach \i/\j in {1/0,1/2,1/1}
        {
 \draw (\i,\j) -- (2,1);
        }
        
        \foreach \i/\j in {5/2,5/0} {
 \draw (\i,\j) -- (4,1);
        }

 \foreach \i/\j in {1/0,1/2,5/2,5/0,4/1,2/1,1/1} {
        \draw[fill=black] (\i,\j) circle (3pt);
        }
    \end{tikzpicture}
\quad
\begin{tikzpicture}[scale=0.5]
        \draw (4,1) -- (2,1);
        \draw (3,1) -- (3,2);
        \foreach \i/\j in {1/0,1/2}
        {
 \draw (\i,\j) -- (2,1);
        }
        
        \foreach \i/\j in {5/2,5/0} {
 \draw (\i,\j) -- (4,1);
        }

 \foreach \i/\j in {1/0,1/2,5/2,5/0,4/1,2/1,3/1,3/2} {
        \draw[fill=black] (\i,\j) circle (3pt);
        }
    \end{tikzpicture}
    \quad 
    \begin{tikzpicture}[scale=0.5]
        \draw (4,1) -- (2,1);
        \foreach \i/\j in {1/0,1/2,1/1}
        {
 \draw (\i,\j) -- (2,1);
        }
        
        \foreach \i/\j in {5/2,5/0} {
 \draw (\i,\j) -- (4,1);
        }

         \draw (3,1) -- (3,2);

 \foreach \i/\j in {1/0,1/2,5/2,5/0,4/1,2/1,1/1,3/1,3/2} {
        \draw[fill=black] (\i,\j) circle (3pt);
        }
    \end{tikzpicture}\quad 
    \begin{tikzpicture}[scale=0.5]
        \draw (5,1) -- (2,1);
        \draw (3,1) -- (3,2);
        \draw (4,1) -- (4,2);
        \foreach \i/\j in {1/0,1/2}
        {
            \draw (\i,\j) -- (2,1);
        }
        \foreach \i/\j in {6/2,6/0}
        {
            \draw (\i,\j) -- (5,1);
        }
        \foreach \i/\j in {1/0,1/2,6/2,6/0,5/1,2/1,3/1,4/1,3/2,4/2}
        {
            \draw[fill=black] (\i,\j) circle (3pt);
        }
    \end{tikzpicture}

 \caption{Series-reduced trees with minimum mean subtree order for $6 \le n \le 10$}
    \label{fig:5nonstar_extremalgraphs}
\end{figure}
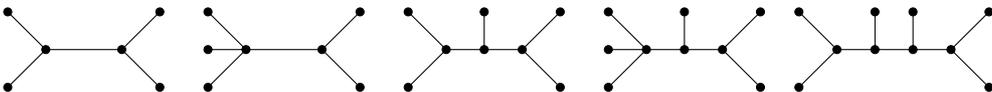

In this subsection, we will prove that the above extremal statement generalizes to $k$-trees without a $k$-clique of degree $2$.

First, we prove the following lemma, which states that every $k$-tree contains a $(k+1)$-clique $C$ that plays the role of a centroid in a tree (a vertex or edge whose removal splits the tree into components of size at most $\frac{n}{2}$). Figure~\ref{fig:oncentroids} is an example of a $2$-tree demonstrating why we must take a $(k+1)$-clique and not a $k$-clique.

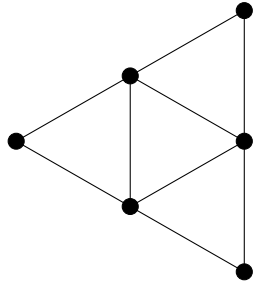
\begin{figure}[h]
    \centering
    \begin{tikzpicture}
        \foreach \i in {120,240,0}
        {
        \draw (\i:1)--(\i+120:1);
        \draw (\i:1)--(\i+60:2);
        \draw (\i:1)--(\i-60:2);
        \draw[fill=black] (\i:1) circle (3pt);
        \draw[fill=black] (\i+60:2) circle (3pt);
        }

    \end{tikzpicture}
    \caption{$2$-tree with a centroid $3$-clique but no centroid $2$-clique}
    \label{fig:oncentroids}
\end{figure}

\begin{lemma}\label{lem:centroid}
    Any $k$-tree $T$ (of order $n \ge k+1$) has a $(k+1)$-clique $C$ such that the order of all components of $T \setminus C$ is at most $\ceilfrac{n-(k+1)}{2}.$
\end{lemma}

\begin{proof}
    Suppose for contradiction that such a $(k+1)$-clique does not exist. Consider the $(k+1)$-clique $C$ for which the largest component of $T \setminus C$ has minimum order over all choices of $C$. By assumption, the largest component of $T \setminus C$ has order $n_0 >\ceilfrac{n-(k+1)}{2}$, and by the construction of a $k$-tree, there exist vertices $u \in C$ and $v$ in the largest component of $T \setminus C$
    such that $ v \cup (C \setminus u)$ forms a $(k+1)$-clique $C'$.
    Now $T \setminus C'$ has components whose sizes are bounded by $\max\left\{ n-n_0-k, n_0-1 \right\}<n_0,$
    which gives a contradiction. 
    Hence $C$ satisfies the stated property.
\end{proof}

The following lemma and its proof are similar to~\cite[Lemma 5.1]{RW18}.

\begin{lemma}\label{lem:local_sim_global}
        For any $k$-tree $T$ without $k$-cliques of degree $2$, there is a $k$-clique $C'$ for which $\mu(T;C') -\mu(T)\le \frac{n^3}{2^{(n-k)/4}}=o_n(1).$
    \end{lemma}
    \begin{proof}
        Take $C$ as in Lemma~\ref{lem:centroid}.
        Let its vertices be $\{u_1, u_2, \ldots, u_{k+1}\}.$
        
        For every component of $T \setminus C$, there is a unique vertex $u_i, 1\le i \le k+1$, such that the component together with $C \setminus u_i$ forms a $k$-tree.
        Now we consider two cases that are handled analogously.

        \textbf{Case 1:} There is some $u_i$ such that the union of components that form a $k$-tree when adding $C'=C \setminus u_i$ has order at least $\frac{n-k}{2}.$

        In this case, we consider the $r \le n-k$ components of $T\setminus C'$ and for every $1 \le i \le r,$ we let $T_i$ be the union of such a component and $C'.$
        
        We then apply the following claim:
         \begin{claim}
            Let $C'$ be a $k$-clique in
        a $k$-tree $T$ of order $n$ with no $k$-cliques of degree $2$.
        Then $T \setminus C'$ has at least 
             $\frac{n-k}2$ simplicial vertices.
         \end{claim}
         \begin{claimproof}
            Consider the $1$-characteristic tree $T'_{C'}$.
            For every $u \in V \setminus C'$, either all the $k$-cliques containing $u$ have degree $1$, in which case $u$ has degree $1$ in $T'_{C'}$, or (at least) one of them has degree at least $3$ and so does $u$ in $T'_{C'}$.
            Thus $T'_{C'}$ is series-reduced and thus has at least $\frac{n-k}2+1$ leaves. The latter implies that $T$ has at least $\frac{n-k}2$ simplicial vertices.
         \end{claimproof}

         From the claim, there are at least $\frac{n-k}{4}-1$ simplicial vertices not belonging to $T_i$ for every $1\le i \le r$. Observe that given any sub-$k$-tree of $T$ containing $C'$, we can map it to its sub-$k$-tree intersection with $T_i$. If two elements of $S(T; C')$ differ only in some subset of $k$-leaves of $\bigcup_{j \neq i} T_j$, then they map to the same element of $S(T_i; C')$. Thus, each element of $S(T_i; C')$ is mapped to at least $2^{\frac{n-k}{4}-1}$ times, and $N(T;C') \ge 2^{\frac{n-k}{4}-1} N(T_i; C')$.

         Using the perfect elimination ordering, every sub-$k$-tree $S$ in $T$ not containing $C$ can be extended in a minimal way into a sub-$k$-tree containing $C'$. 
         Furthermore, by considering the $1$-characteristic tree $T'_{C'}$ of $C'$, it is clear that there are no more than $\abs{T_i}-k \le \ceilfrac{n-(k+1)}{2} \le \frac n2$ $k$-trees that extend to the same tree. Here we have used Lemma~\ref{lem:centroid}. Thus, if we define a map from $\overline{S}(T_i; C')$ to $S(T_i; C')$ using the minimal extension, every element of $S(T_i; C')$ is mapped to at most $\frac{n}{2}$ times. 

         Putting the previous two observations together, we have that the number of sub-$k$-trees containing $C'$ is $N(T;C') \ge 2^{\frac{n-k}{4}-1} N(T_i; C') \ge \frac 1n 2^{\frac{n-k}4} \overline N(T_i; C').$
        
        By summing over all $i,$
        we obtain that 

    $$r N(T;C') \ge \frac 1n 2^{\frac{n-k}4} \sum_{i=1}^r \overline N(T_i; C') = \frac 1n 2^{\frac{n-k}4} \overline N(T;C').$$

    Since $r \le n,$ we have that $N(T;C') \ge \frac 1{n^2} 2^{\frac{n-k}4} \overline N(T;C').$ 
    This implies that $\frac{\mu(T)}{\mu(T; C')} \ge 1- \frac {n^2}{2^{\frac{n-k}4}}.$
    The result follows now from $\mu(T; C') \le n.$

        \textbf{Case 2:} For every $u_i$ the union of components in $T \setminus C$ that form a $k$-tree when adding $C \setminus u_i$ has order smaller than $\frac{n-k}{2}.$

        Let $T_i$, $1 \le i \le k+1$, be the $k$-trees obtained above when adding $C \setminus u_i$.
        By the claim above, for each $i$ there are at least $\frac{n-k}{4}-1$ $k$-leaves not belonging to $T_i$. Thus, the same computations apply and in particular we have that $N(T;C') \ge \frac 1{n^2} 2^{\frac{n-k}4} \overline N(T;C')$ as before.
        Now for $C'=C \setminus u_1,$ we conclude by inclusion monotonicity~\cite[Thm.~33]{LX23} that 
        $\mu(T;C') -\mu(T) \le \mu(T;C) -\mu(T)\le \frac{n^3}{2^{(n-k)/4}}.$
    \end{proof}

\begin{proof}[Proof of Theorem~\ref{thr:prob6}, lower bound]
    Take a $k$-clique $C'$ which satisfies Lemma~\ref{lem:local_sim_global}.
    Let $T'_{C'}$ be the $1$-characteristic tree of $T$ with respect to $C'.$
    By~\cite[Thm.~33]{LX23} and~\cite[Lem.~12]{John14}, we conclude that
    $\mu(T;C')=\mu(T'_{C'}; C')+k-1 \ge \frac{n+k}{2}+\frac{i -1}{10},$
    where $i$ is the number of internal vertices in $T'_{C'}.$
    By Lemma~\ref{lem:local_sim_global}, we conclude.
    
For sharpness, observe that if $T$ is not a $k$-star, we have $i \ge 2$ and the lower bound inequality is strict.
When $T$ is a $k$-star, it contains no $k$-cliques of degree $2$ provided that $n>k+2$.
As computed in~\cite{SO18} (page 64), $\mu(T)=\frac{R(T)}{N(T)} = \frac{(n+k)2^{n-k-1}}{2^{n-k}+(n-k)k} =\frac{n+k}{2}-o_n(1).$
\end{proof}

\subsection{The upper bound}

In this subsection, we generalize to $k$-trees the statement that a series-reduced tree has average subtree order at most $\frac{3n}{4}$ by giving a lower bound for the number of $k$-leaves and proving that $k$-leaves belong to at most half of the sub-$k$-trees.
This idea was also used in~\cite{John14}.
Note that the upper bound is slightly larger than $\frac{3n}{4}$ for larger $k$, which intuitively can be explained by the fact that the smallest sub-$k$-tree already has $k$ vertices, and more precisely the vertices in the base $k$-clique will all be major vertices.

\begin{proof}[Proof of Theorem~\ref{thr:prob6}, upper bound]
 Our upper bound will come from the observation that $\mu(T) = \sum_{v \in T} p(v)$ where $p(v)$ is the fraction of sub-$k$-trees containing $v$. We will specifically consider when $v$ is a $k$-leaf and bound the corresponding terms in the summation.
 
 We first prove that the $1$-characteristic tree $T'_{C}$ of a $k$-tree $T$ without $k$-cliques of degree $2$ is a series-reduced tree, for any $k$-clique $C$ of $T$. Indeed, given a $k$-clique $C$, there is either exactly one vertex adjacent to $C$ or at least $3.$ As such, the degree of $C$ in $T'_{C}$ is not $2.$
For any other vertex $v \in T$, either it is a leaf in $T'_{C}$, or at least one other vertex has been added to a $k$-clique $C'$ containing $v$ in $T$. In the latter case, since $T$ has no $k$-cliques of degree 2, there must be at least two vertices other than $v$ joined to $C'$ and thus the degree of $v$ in $T'_{C}$ is at least $3$. 
Since $T'_{C}$ has at least $\frac{n-k+1}{2}+1$ leaves, $T$ contains at least this many $k$-leaves (here one must also observe that if $C$ has degree $1$ in $T'_{C}$, some vertex of $C$ is simplicial in $T$).

 Now fix a $k$-leaf $v.$ Since $n \ge k+2,$ there is a vertex $u$ (different from $v$) such that $N(v)\cup \{u\}$ spans a $K_{k+1}$. 
 Define a function $f$ on sub-$k$-trees containing $v$ such that $f(C') = (C' \setminus \{v\}) \cup \{u\}$ for a $k$-clique $C'$ and $f(T') = T' \setminus v$ otherwise. We can check that $f$ maps $k$-cliques to $k$-cliques and is in fact an injection from sub-$k$-trees containing $v$ to sub-$k$-trees not containing $v$. Indeed, because $v$ is a $k$-leaf, $C' \cup \{u\}$ is not a $(k+1)$-clique and thus not a sub-$k$-tree. Hence there does not exist a $(k+1)$-clique $C''$ for which $C'' \setminus \{v\} = (C' \setminus \{v\}) \cup \{u\}$.

 This implies that every $k$-leaf belongs to at most half of the sub-$k$-trees in $T$.
 Remembering that there are at least $\frac{n-k+3}{2}$ $k$-leaves, the global mean order of $T$ is then 
 $$\mu(T) = \sum_{v\ \text{non-}k\text{-leaf}} p(v) + \sum_{v\ k\text{-leaf}}p(v) \leq n-\frac 12 \cdot \frac{n-k+3}{2} = \frac{3n+k-3}{4}$$

 For sharpness, let $n=2s+3-k$ for an integer $s.$ We construct $T$ by first constructing a caterpillar $T'$ which consists of a path-type $k$-tree $P_s^{k+1}$ on vertices $v_1, v_2, \ldots, v_s$, for which $k+1$ consecutive vertices form a clique,
 and adding a $k$-leaf connected to every $k$ consecutive vertices.
  To obtain $T$, we extend $T'$ by adding two $k$-leaves, which are connected to $\{v_1, \dots, v_k\}$ and $\{v_{s-k+1}, \dots, v_s\}$ respectively. Note that $T$ has a ``stem'' of $s$ vertices and the number of $k$-leaves is $\ell+2=(s-k+1)+2=s-k+3$.
   For every $1 \le i \le s-k$, there are $i$ sub-$k$-trees of the stem each of order $s+1-i$.
 Each of these can be extended by adding any subset of the $\ell+1-i$ neighboring $k$-leaves, or even one or two more if some of the end-vertices of the stem are involved.

Now we can compute that 
\begin{align*}
    N(T)&=(k(n-k)+1-\ell) + 2^{\ell +2} + 2\sum_{i=2}^{\ell} 2^i + \sum_{i=3}^{s-k+1} (i-2)\cdot 2^{\ell+1-i}\\
    &\sim 9\cdot 2^{\ell}.
\end{align*}
The first expression $k(n-k)+1-\ell$ counts the number of simplicial $k$-cliques different from the ones consisting of $k$ consecutive vertices in the stem. $2^{\ell+2}$ is the number of sub-$k$-trees containing the whole stem. The third term counts the number of sub-$k$-trees containing $v_1$ or $v_s$ but not both and at least $k$ vertices of the stem. The last summation counts the sub-$k$-trees containing at least $k$ vertices of the stem and none of $v_1$ and $v_s.$
We can also compute $R(T)$ by summing the total size of the respective sub-$k$-trees.

\begin{align*}
 R(T) &=(k(n-k)+1-\ell)k
  +2^{\ell+2}\left(s+\frac{\ell+2}{2}\right)
   +2\sum_{i=2}^{\ell} 2^i\left( k+i-2+\frac i2 \right)\\
 &+ \sum_{i=3}^{s-k+1} (i-2)\cdot 2^{\ell+1-i}\left( s+1-i +\frac{\ell+1-i}{2} \right)\\
 &\sim 2^{\ell+2}\left(s+\frac{\ell+2}{2}\right)+2^{\ell+2}(k-2) + 3\cdot 2^{\ell+1}(\ell-1) + 2^{\ell} \left( s+1 +\frac{\ell+1}{2} \right) -15\cdot 2^{\ell -1}\\
 &=\left( 5s+\frac{17}2 \ell +4k-16\right)2^{\ell}\\
 &=\left( \frac{27}{4}n +\frac{9}{4}k-\frac{111}{4}\right)2^{\ell}.
 \end{align*}

Finally, we conclude that $\mu(T)=\frac{R(T)}{N(T)}\sim \frac{3}{4}n +\frac{1}{4}k-\frac{37}{12},$ which is only $\frac{7}{3}$ away from the upper bound. These computations have also been verified in ~\url{https://github.com/StijnCambie/AvSubOrder_ktree/blob/main/M_comb_ktree.mw}.

\end{proof}

\begin{figure}[h]

\begin{minipage}[b]{.45\linewidth}
\begin{center}
 
\begin{tikzpicture}[thick]

\foreach \x in {0, 1,2}{
\draw[fill] (120*\x:1) circle (0.15);
}

\foreach \y in {0,1,2,3,4,5}
{
\draw[fill] (60*\y:2) circle (0.15);
}

\foreach \y in {0,1,2}
{
\draw[rotate=120*\y] (0:0.8) ellipse (2cm and 1.5cm);
}

\foreach \y in {0,1,2}
{
\draw[rotate=120*\y+60, blue] (0:0.8) ellipse (2cm and 1.5cm);
}

\end{tikzpicture}\\
\subcaption{A $k$-star}
\label{fig:doublebroom}
\end{center}
\end{minipage}
\quad\begin{minipage}[b]{.45\linewidth}
\begin{center}
 \begin{tikzpicture}

\foreach \y in {0,2}{
\draw[fill] (\y+2,1.5) circle (0.15);
\draw[red] (2+\y,0.6) ellipse (1.5cm and 1.25cm);
}

\foreach \y in {0,1,2,3}{
\foreach \x in {0, 1,2,3}{

\draw[fill] (\x+\y,0) circle (0.15);
}
\draw (1.5+\y,0) ellipse (2cm and 0.75cm);
}

 \end{tikzpicture}\\
\subcaption{A $k$-caterpillar}
\label{fig:widespider}
\end{center}
\end{minipage}

 \caption{Sketch of $k$-trees with mean sub-$k$-tree order roughly $\frac n2$ and $\frac{3n}{4}$ for $k=3$}
 \label{fig:star1/2+cat3/4}
\end{figure}
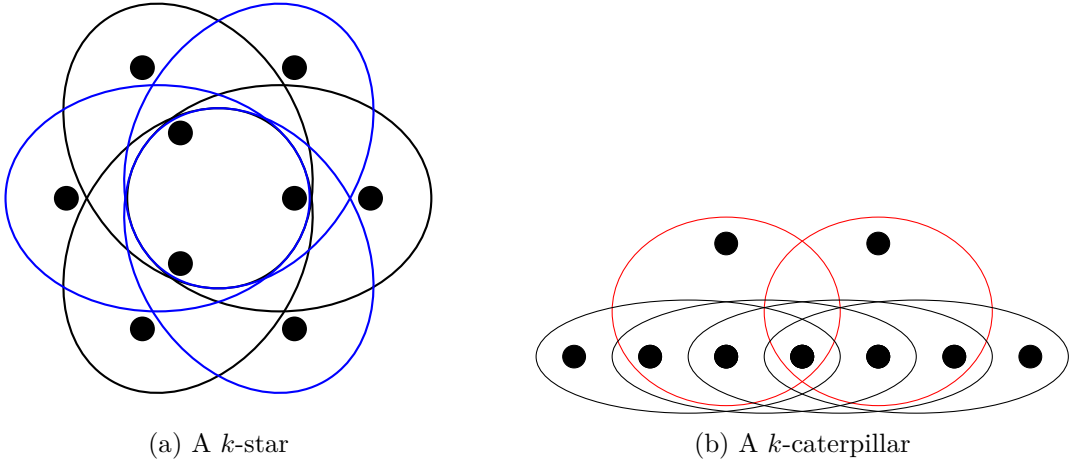

\section{The maximum local mean order}\label{sec:prob4}

In this section, we prove 
\maxlocalmean* 

We consider the $k$-clique in a $k$-tree $T$ for which the local mean order is greatest. Our aim is to show that its degree cannot be too large, specifically at most $2$. To do so, we prove that if a $k$-clique $C$ has degree at least $2$, then there is a neighboring $k$-clique whose local mean order is not smaller, with strict inequality when $C$ has degree at least $3$.  From this, it will follow that any $k$-clique attaining the maximum has degree at most $2$ and there is always at least one $k$-clique with degree $1$ in which the maximum is attained. 

The proof requires some technical inequalities, which we prove first. We start with a generalization of~\cite[Lemma 2.1]{WW16} to the $k$-tree case as follows.

\begin{lemma}\label{lem:bound_R}
For a sub-$k$-tree $T$ and a $k$-clique $C$, we have
$$R(T; C) \leq \frac{N(T; C)^2 + (2k-1)N(T; C)}{2}.$$
Equality holds if and only if $T$ is a path-type $k$-tree and $C$ is a simplicial $k$-clique. 
\end{lemma}
 
\begin{proof}
The statement is clearly true when $T = C$, so assume $\abs T > k$. 
Observe that every sub-$k$-tree $T'$ of $T$ containing $C$ has a vertex $v$ (not belonging to $C$) which is simplicial within $T'$ (consider a perfect elimination order where $C$ is taken as the base $k$-clique).
This implies that $T'\backslash v$ is a sub-$k$-tree containing $C$ as well. 

If there exists a sub-$k$-tree containing $C$ of order $\ell > k$, then the above implies that there must also exist a sub-$k$-tree of order $\ell - 1$ containing $C$. Thus, if we list the sub-$k$-trees in $S(T; C)$ from smallest to largest order, we see that
$$R(T; C) \leq k+(k+1)+\cdots + (k+N(T; C)-1) = \frac{1}{2}(N(T; C)^2+(2k-1)N(T; C))$$
as desired. 

The equality case is clear, since every sub-$k$-tree $T'\not=C$ of $T$ must have exactly one $k$-leaf not belonging to $C$, and equality is attained when $T$ is a path-type $k$-tree.
\end{proof}

Using the previous lemma, we can now bound the local mean order in terms of the number of sub-$k$-trees.

\begin{lemma}\label{lem:mu_ifv_N}
For a $k$-tree $T$ and one of its $k$-cliques $C$, we have
 $$ k+\frac{\log_2 {N(T; C)}}{2} \le \mu(T; C) \leq \frac{N(T; C) + (2k-1)}{2}.$$
 The minimum occurs if and only if $T$ is a $k$-star and $C$ its base $k$-clique. The maximum is attained exactly when $T$ is a path-type $k$-tree and $C$ is simplicial.
\end{lemma}

\begin{proof}
 The upper bound follows immediately from Lemma~\ref{lem:bound_R} by dividing both sides of the inequality by $N(T; C)$. The equality cases are the same. 
 
 By~\cite[Thm.~12]{SO18} we have $\mu(T; C) \ge \frac{\abs T +k}{2}.$
 Since a sub-$k$-tree is determined by its vertices, we also have 
 \begin{equation}\label{eqn:nbound}
 N(T; C)\le 2^{\abs T-k}.
 \end{equation}
 Combining these two inequalities, we get $\mu(T; C) \ge \frac{\log_2 {N(T; C)}}{2}+k$.

 To attain the lower bound, equality must hold for \eqref{eqn:nbound}. This is the case if and only if $T$ is a $k$-star and $C$ its base $k$-clique. Indeed, for a $k$-star of order $n$ with base clique $C$, we have $N(T;C)=2^{n-k}$ and $\mu(T;C)=\frac{n+k}2$. 
 In the other direction, every vertex together with $C$ needs to form a $k$-tree and thus a $(k+1)$-clique, which is possible only for a $k$-star. 
\end{proof}

We will also make use of the following elementary inequality, which can be considered as the opposite statement of the inequality between the arithmetic mean and geometric mean (AM-GM).

\begin{lemma}\label{lem:trivialinequality}
 Let $x_1, x_2, \ldots, x_n \in \mathbb R_{\geq 1}$, and let $P = \prod_i x_i$.
 Then $\sum_i x_i \le P+(n-1).$
 Furthermore, equality is attained if and only if $x_i = P$ for some $i$ and $x_j = 1$ for all $j \neq i$.
\end{lemma}

\begin{proof}
 This can be proven in multiple ways.
 The most elementary way is to observe that if $x_i,x_j>1$, then $x_ix_j + 1>x_i+x_j$ since $(x_i-1)(x_j-1)>0.$
 Repeating this with pairs of elements which are strictly larger than $1$ gives the result.
 Alternatively, one could consider the variables $\alpha_i= \log(x_i) \ge 0$. Since their sum is the fixed constant $\log P$ and the exponential function is convex, as a corollary of Karamata's inequality $\sum_i \exp(\alpha_i)$ is maximized when all except one are equal to $0.$
\end{proof}

We are now ready to prove Theorem~\ref{thr:prob4}.
Recall that $T$ can be decomposed into $C$ and $k$-trees $T_{1,1},\ldots,T_{d,k}$ rooted at $C_{1,1},\ldots,C_{d,k}$ respectively that are pairwise disjoint except for the vertices of the cliques $C_{i,j}$. Let $B_i$ denote the $(k+1)$-clique that contains $C$ as well as $C_{i,1},\ldots,C_{i,k}$, and let $v_i$ be the vertex of $B_i \setminus C$. Finally, set $N_{i,j} = N(T_{i,j};C_{i,j})$ and $\mu^\bullet_{i,j} = \mu^\bullet(T_{i,j};C_{i,j})$.

\begin{proof}[Proof of Theorem~\ref{thr:prob4}]
We first observe that
\begin{equation}\label{rec:number-of-sub-$k$-trees}
N(T; C) = \prod_{i=1}^d \left( 1 + \prod_{j=1}^k N_{i,j} \right).
\end{equation}
This is because a sub-$k$-tree $S$ of $T$ containing $C$ is specified as follows: given $1 \leq i \leq d$, choose the sub-$k$-tree intersection of $S$ with $T_{i,1}, \dots, T_{i,k}$. There are $\prod_{j=1}^k N_{i,j}$ ways to do this if $v_i \in S$ and one if $v_i \notin S$. We do this independently for each $i$, resulting in the product above.

Next, we would like to express the local mean order at $C$ in terms of the quantities $N_{i,j}$ and $\mu^\bullet_{i,j}$: it is given by
\begin{equation}\label{eq:mu_bullet(TC)}
\mu^\bullet(T; C) = \sum_{i=1}^d \frac{1 + \sum_{j=1}^k \mu^\bullet_{i,j}}{1+\prod_{j=1}^k N_{i,j}^{-1}} .\end{equation}
To see this, note that we can interpret $\mu(T; C)$ as the expected size of a random sub-$k$-tree chosen from $S(T;C)$, which can then be written as the sum of the expected sizes of the intersection with each component $T_{i,j}$ in the decomposition. We have that $\frac{1}{1+\prod_{j=1}^k N_{i,j}^{-1}} = \frac{\prod_{j=1}^k N_{i,j}}{1+\prod_{j=1}^k N_{i,j}}$ is the probability that a randomly chosen sub-$k$-tree of $T$ that contains $C$ also contains $v_i$ (by the same reasoning that gave us~\eqref{rec:number-of-sub-$k$-trees}). Once $v_i$ is included, it adds $1$ to the number of vertices, and an average total of $\sum_{j=1}^k \mu^\bullet_{i,j}$ is added from the extensions in $T_{i,1},\ldots,T_{i,k}$. 

Without loss of generality, we can assume that $N_{1,1} = \min_{i,j} N_{i,j}$. We want to compare $\mu^\bullet(T; C)$ to $\mu^\bullet(T; C_{1,1})$ and prove that $\mu^\bullet(T; C_{1,1}) \geq \mu^\bullet(T; C)$ provided that $d \geq 2$, with strict inequality if $d > 2$. Observe that this will be enough to prove our claim: no clique with degree greater than 2 can attain the maximum local mean order, and starting from any clique, we may repeatedly apply the inequality above to obtain a sequence of neighboring cliques whose local mean orders are weakly increasing and the last of which is a degree-1 $k$-clique (more precisely, for the first step replacing $C$ with $C':= C_{1,1}$ and considering the decomposition $\{C'_{i,j}\}$ and $\{T'_{i,j}\}$ with respect to $C'$, the clique adjacent to $C'$ with equal or larger $\mu^\bullet$ cannot be the original clique $C$ as $C$ will not correspond to $\min_{i,j} N'_{i,j}$. So in repeatedly applying the inequality, we will obtain a sequence of distinct cliques which must terminate but can only terminate once we have reached a clique of degree 1.)

Let us first express $\mu^\bullet(T; C_{1,1})$ in terms of the $N_{i,j}$ and $\mu^\bullet_{i,j}$ as well. First, we have
$$N(T; C_{1,1}) = N_{1,1} + \prod_{j=1}^k N_{1,j} \prod_{i=2}^d \left( 1 + \prod_{j=1}^k N_{i,j} \right).$$
The reasoning is similar to~\eqref{rec:number-of-sub-$k$-trees}: there are $N_{1,1}$ sub-$k$-trees that contain $C_{1,1}$, but not the full $(k+1)$-clique $B_1$, and the remaining product counts sub-$k$-trees containing $B_1$. We also have
$$\mu^\bullet(T; C_{1,1}) = \frac{N_{1,1}\mu^\bullet_{1,1}}{N(T; C_{1,1})} + \left(1 - \frac{N_{1,1}}{N(T; C_{1,1})} \right) \left( 1 + \sum_{j=1}^k \mu^\bullet_{1,j} + \sum_{i=2}^d \frac{1 + \sum_{j=1}^k \mu^\bullet_{i,j}}{1+\prod_{j=1}^k N_{i,j}^{-1}} \right),$$
using the fact that $\frac{N_{1,1}}{N(T; C_{1,1})}$ is the probability that a random sub-$k$-tree containing $C_{1,1}$ does \emph{not} contain $B_1$. Let us now take the difference $\mu^\bullet(T; C_{1,1}) - \mu^\bullet(T; C)$: we have, after some manipulations,
\begin{align*}
 \mu^\bullet(T; C_{1,1}) - \mu^\bullet(T; C) &= \frac{N_{1,1}}{N(T; C_{1,1})} \left(\mu^\bullet_{1,1} - \mu^\bullet(T; C) \right) \\
 &\quad + \left(1 - \frac{N_{1,1}}{N(T; C_{1,1})} \right) \left(1 + \sum_{j=1}^k \mu^\bullet_{1,j} - \frac{1 + \sum_{j=1}^k \mu^\bullet_{1,j}}{1+\prod_{j=1}^k N_{i,j}^{-1}}\right)\\
 &= \frac{N_{1,1}}{N(T; C_{1,1})} \left(\mu^\bullet_{1,1} - \mu^\bullet(T; C) \right) + \left(1 - \frac{N_{1,1}}{N(T; C_{1,1})} \right) \left(\frac{1 + \sum_{j=1}^k \mu^\bullet_{1,j}}{1+\prod_{j=1}^k N_{1,j}}\right).
\end{align*}

Now let $F = F(k):=\frac{1 + \sum_{j=1}^k \mu^\bullet_{1,j}}{1+\prod_{j=1}^k N_{1,j}}$. We want to show that $\mu^\bullet(T; C_{1,1}) - \mu^\bullet(T; C) \geq 0$, i.e., that
$$\frac{N_{1,1}}{N(T; C_{1,1})}\mu^\bullet_{1,1} + F \geq \frac{N_{1,1}}{N(T; C_{1,1})} \left(\mu^\bullet(T; C) + F \right).$$
Equivalently,
\begin{equation}\label{eq:ineq}
    \mu^\bullet_{1,1} + \frac{N(T; C_{1,1})}{N_{1,1}} \cdot F \geq \mu^\bullet(T; C) + F.
\end{equation}
Using the previous computations, we know that
$$\frac{N(T; C_{1,1})}{N_{1,1}} = 1 + \prod_{j=2}^k N_{1,j}\prod_{i=2}^d \left(1+\prod_{j=1}^k N_{i,j}\right).$$
So the left-hand side of \eqref{eq:ineq} is equal to
$$\mu^\bullet_{1,1} + F + \frac{\prod_{j=2}^kN_{1,j}}{1+\prod_{j=1}^k N_{1,j}} \left(1+\sum_{j=1}^k \mu^\bullet_{1,j}\right) \prod_{i=2}^d \left(1+\prod_{j=1}^k N_{i,j}\right).$$
We can subtract $F$ from both sides and use \eqref{eq:mu_bullet(TC)} to replace $\mu^\bullet(T; C)$. Taking into account that $\mu^\bullet_{1,1} \ge 0$, it is sufficient to prove
\begin{equation}\label{eq:main_ineq} \frac{\prod_{j=2}^kN_{1,j}}{1+\prod_{j=1}^k N_{1,j}}\left(1+\sum_{j=1}^k \mu^\bullet_{1,j}\right) \prod_{i=2}^d \left(1+\prod_{j=1}^k N_{i,j}\right) \geq \sum_{i=1}^d \frac{1+\sum_{j=1}^k \mu^\bullet_{i,j}}{1+\prod_{j=1}^k N_{i,j}^{-1}}.
\end{equation}

We first prove \eqref{eq:main_ineq} for $d=2$,
in which case it can be rewritten as 

$$\frac{\prod_{j=1}^k N_{1,j}}{N_{1,1}(1+\prod_{j=1}^k N_{1,j})}\left(1+\sum_{j=1}^k \mu^\bullet_{1,j}\right) \left(1-N_{1,1}+\prod_{j=1}^k N_{2,j}\right) \geq \frac{\prod_{j=1}^k N_{2,j}}{1+\prod_{j=1}^k N_{2,j}} \left(1+\sum_{j=1}^k \mu^\bullet_{2,j}\right).$$

Write $\prod_{j=1}^k N_{1,j}= N_{1,1}^{k-1} y$ and $\prod_{j=1}^k N_{2,j}= N_{1,1}^{k-1} z$ where $y, z \ge N_{1,1}.$
By Lemma~\ref{lem:mu_ifv_N}, we have $\mu^\bullet_{1,j} \geq \frac12 \log_2 N_{1,j}$ and $\mu^\bullet_{2,j} \leq \frac12 (N_{2,j} - 1)$. Applying Lemma~\ref{lem:trivialinequality} to the numbers $\frac{N_{2,j}}{N_{1,1}}$, $1 \leq j \leq k$, gives us that $\sum_{j=1}^k N_{2,j} \leq (k-1)N_{1,1} + z$. Hence we find that it suffices to prove that

$$\frac{y\left(1-N_{1,1}+N_{1,1}^{k-1}z\right)}{N_{1,1}(1+ N_{1,1}^{k-1}y)}\left(1+\frac{(k-1) \log_2{ N_{1,1}} + \log_2{y}}2 \right) \geq \frac{z}{1+N_{1,1}^{k-1}z} \left(1+\frac{ (k-1)N_{1,1} +z-k}{2} \right).$$

We note that the left-hand side is strictly increasing in $y$ and the right-hand side is independent of $y,$ which implies that it is sufficient to prove the equality when $y=N_{1,1}.$ That is, we want to prove that for every $z \ge y \ge 1$

\begin{equation}\label{eq:d=2_ifv_yz}
    \frac{1-y+y^{k-1}z}{1+y^k}\left( 1+ \frac k2 \log_2 y \right) \ge \frac{z}{1+y^{k-1}z}\left(1+\frac{ (k-1)y +z-k}{2} \right).
\end{equation}

If $y= N_{1,1} = 1$, \eqref{eq:d=2_ifv_yz} reduces to the equality $\frac z2 = \frac z2.$

If $y= N_{1,1} \ge 2$, $z\ge y$ and $k \ge 2$ imply that $2+(k-1)y+z-k \le kz$. Together with $\log_2 {y} \ge 1$, we conclude that it is sufficient to prove that

$$\frac{1-y+y^{k-1}z}{1+y^k} \frac {k+2}2  \ge \frac{z}{1+y^{k-1}z}\frac{ kz}{2}$$
which is equivalent to
$$ (1-y+y^{k-1}z)(k+2)(1+y^{k-1}z) \ge (1+y^k)kz^2.$$

For $k=2$, the difference between the two sides is an increasing function in $y$, and for $y=2$ it reduces to $3z^2-2 \ge 0$, which holds.

For $k \ge 3$, the inequality is immediate, using that 
$y^{2(k-1)}z^2 \ge (1.5y^k+1)z^2 \ge y^kz + y^kz^2+z^2$ (remember that $z \ge y \ge 2$).

Once \eqref{eq:main_ineq} has been verified for $d=2$, we can apply induction to prove it for $d \geq 3$. Let $C=\frac{\prod_{j=2}^kN_{1,j}}{1+\prod_{j=1}^k N_{1,j}}\left(1+\sum_{j=1}^k \mu^\bullet_{1,j}\right),$ $g_i= 1+\prod_{j=1}^k N_{i,j}$ and $f_i=\frac{1+\sum_{j=1}^k \mu^\bullet_{i,j}}{1+\prod_{j=1}^k N_{i,j}^{-1}}.$ We can then rewrite \eqref{eq:main_ineq} as

$$C \prod_{i=2}^d g_i \ge \sum_{i=1}^d f_i$$

By the induction hypothesis, we have
$$\frac{C}{g_m} \prod_{i=2}^d g_i \ge \left(\sum_{i=1}^d f_i\right)- f_m$$ for every $2 \le m \le d.$
Summing over $m$, we obtain that 
$$C \prod_{i=2}^d g_i \left( \sum_{m=2}^d \frac 1{g_m} \right) \ge f_1+(d-2)\sum_{i=1}^d f_i.$$
Since we have $g_m \ge 2$ for every $m$, the conclusion now follows as $$\sum_{m=2}^d \frac 1{g_m} \le \frac{(d-1)}{2} \le d-2,$$
and since $f_1>0$, inequality~\eqref{eq:main_ineq} is even strict in the case that $d > 2$. 

We conclude that if $C$ has degree $d \ge 2$, then
$\mu^\bullet(T; C) \le \mu^\bullet(T; C_{1,1}).$
Equality can only be attained when $d=2$. Looking back over the proof of~\eqref{eq:main_ineq}, we also see that for equality to hold, all $N_{i,j}$ except for $N_{2,2}$ (up to renaming) must be equal to $1$, and due to Lemma~\ref{lem:mu_ifv_N}, $T_{2,2}$ has to be a path-type $k$-tree.

For $n\ge k+3$, we compute that in a path-type $k$-tree $T$ the maximum among the degree-$1$ $k$-cliques is attained by a central one, which implies that no degree-$2$ $k$-cliques can be extremal.
Let $B_1$ be the unique $(k+1)$-clique containing $C$, and assume that $T \setminus B_1$ contains components of size $a$ and $b.$ Thus $a+b=n-(k+1).$
Then $\mu(T; C)= \frac{k+(a+1)(b+1)\frac{n+k}{2}}{(a+1)(b+1)+1}=\frac{n+k}{2}-\frac{n-k}{2((a+1)(b+1)+1)}$, and this is maximized if and only if $\abs{a-b}\le 1.$
The latter can also be derived from considering the $1$-characteristic tree.
This is illustrated in Figure~\ref{fig:nonsimplicialmaximizer} below.
We emphasize that the maximizing $k$-clique of degree $1$ is not simplicial.

When $n \le k+2,$ every $k$-clique has the same local mean sub-$k$-tree order, and so the (unique) degree-$2$ $k$-clique when $n=k+2$ is the only case where equality can occur at a $k$-clique with degree $2.$
This concludes the proof of Theorem~\ref{thr:prob4}.
\end{proof}

\begin{figure}[h]
    \centering
    \begin{tikzpicture}
        \draw (1,1.73205)--(0,0)--(4,0)--(3,1.73205);
        \draw (1,1.73205)--(-1,1.73205)--(0,0);
        \draw (3,1.73205)--(5,1.73205)--(4,0);
        \draw[blue] (1,1.73205)--(3,1.73205);
        \draw (3,1.73205)--(2,0)--(1,1.73205);
        \foreach \i in {0,2,4}
        {
        \draw[fill=black] (\i,0) circle (3pt);
        }
        \foreach \i in {-1,5,1,3}
        {
        \draw[fill=black] (\i,1.73205) circle (3pt);
        }
        \node [label= $C$] (0) at (2,1.73) {};
    \end{tikzpicture} 
    \quad
    \begin{tikzpicture}
        \draw[fill=black] (2,1) circle (3pt);
        \draw (2,1)--(2,0);
        \draw (0,0)--(4,0);
       \foreach \i in {0,1,2,3,4}
        {
        \draw[fill=black] (\i,0) circle (3pt);
        }
        \node [label= $C$] (0) at (2,1)  {0};
    \end{tikzpicture}
    \caption{$2$-path for which the maximum local mean is attained in a non-simplicial clique $C$, with its $1$-characteristic tree $T'_C$}
    \label{fig:nonsimplicialmaximizer}
\end{figure}
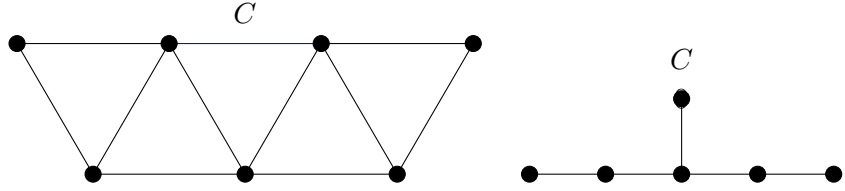

\section{Bounding local mean order by global mean order}\label{sec:prob2}

In this section, we prove
\localboundedbyglobal*

As before, given a $k$-tree $T$ and a $k$-clique $C$ in $T$, we utilize the decomposition of $T$ into $C$ and $k$-trees $T_{1,1},\ldots,  T_{d,k}$ rooted at $C_{1,1},\ldots, C_{d,k}$. Set $N_{i,j}=N(T_{i,j};C_{i,j})$, $\overline{N}_{i,j}=\overline{N}(T_{i,j};C_{i,j})$,
$\mu_{i,j}=\mu(T_{i,j};C_{i,j})$, $\overline{\mu}_{i,j}=\overline{\mu}(T_{i,j};C_{i,j})$, $\mu^\bullet_{i,j}=\mu^\bullet(T_{i,j};C_{i,j})$ and $R_{i,j}=R(T_{i,j};C_{i,j})$, $\overline{R}_{i,j}=\overline{R}(T_{i,j};C_{i,j})$. 

Since a sub-$k$-tree not containing $C$ needs to be a sub-$k$-tree of some $k$-tree $T_{i,j},$
we have
$$\overline{N}(T;C)=\sum_{i=1}^d \sum_{j=1}^k(N_{i,j}+\overline{N}_{i,j})$$
and
$$\overline{R}(T;C)=\sum_{i=1}^d \sum_{j=1}^k(R_{i,j}+\overline{R}_{i,j}).$$

Following the proof for trees, we show 
\begin{lemma}\label{lem:RgeqN}
For any $k$-tree $T$ and $k$-clique $C\in T$,
$$R(T;C)>\overline{N}(T;C).$$
\end{lemma}

\begin{proof}
    Assume to the contrary that there exists a minimum counterexample $T$. Since the statement is true when $T=C,$ we have $\abs T >k$ and we can consider the decomposition as before.
    
    Note that if $N_{i,j}=1,$ we have that $\overline N_{i,j}=0$, and otherwise we have $\overline N_{i,j} \le R_{i,j}=\mu_{i,j}N_{i,j}=(k+\mu^\bullet_{i,j})N_{i,j}.$

    We can rewrite $R(T;C)-\overline{N}(T;C)$ as 
    $$N(T;C) \left( k+ \mu^\bullet(T; C)\right)-\overline{N}(T;C).$$
    Expanding using~\eqref{rec:number-of-sub-$k$-trees} and \eqref{eq:mu_bullet(TC)},
    we note that the coefficient of $\mu^\bullet_{i,j}$ is at least equal to $N_{i,j}$.
    As such, it is sufficient to prove the inequality with $\mu^\bullet_{i,j}=0$ and $\overline N_{i,j}$ bounded by $k N_{i,j}.$

    Let $f$ be a function on the positive integers defined by $f(x)=\begin{cases}
        1 & \text{ if } x=1,\\
        1+kx & \text{ if } x>1.
    \end{cases}$

    We now want to prove that

    \begin{equation}\label{eq:crucial_forRgeNbar}
        \left( k+ \sum_{i=1}^d \frac{1}{1+\prod_{j=1}^k N_{i,j}^{-1}}  \right)N(T; C) \ge \sum_{i=1}^d \sum_{j=1}^k f(N_{i,j}).
    \end{equation}

    When $d=1,$ this becomes
    $k+(k+1)\prod_{j=1}^k N_{1,j} \ge \sum_{j=1}^k f(N_{1,j}).$
    When increasing a value $N_{1,j}$ which is at least equal to $2$, the left-hand side increases more than the right-hand side.
    As such, it is sufficient to consider the case where $a$ of the terms $N_{1,j}$ equal $2$, while the other $k-a$ terms equal $1$.
    In this case, the desired inequality holds since $k+(k+1)2^a > k+2ak$ for every integer $0 \le a \le k$.
    
    Next, we consider the case $d\ge 2.$
    In this case,
    when $N_{i,j}$ increases by $1$, the left-hand side of~\eqref{eq:crucial_forRgeNbar} increases by at least $2k$ and the right-hand side by at most $2k.$
    When all $N_{i,j}$ are equal to $1$, the conclusion follows from $\left(k+ \frac d2\right)2^d >k 2^d  > dk$.
\end{proof}

We now bound the local mean order by the global mean order.

\begin{proof}[Proof of Theorem~\ref{thr:prob2}]
    Let $T$ be a $k$-tree and $C$ a $k$-clique in $T$. We want to prove that
    $$\mu(T; C) < 2 \mu(T).$$
    We proceed by induction on the number of vertices in $T$. Note first that the inequality is trivial if $|T| \leq 2k$: since the mean is taken over sub-$k$-trees, which have at least $k$ vertices each, we have $\mu(T) \geq k$. On the other hand, we clearly have $\mu(T;C) \leq |T|$, and both inequalities hold with equality only if $|T|=k$.
    
    We thus proceed to the induction step, and assume that $\abs T > 2k$.
    We have two cases with respect to $C$.
    
    \textbf{Case 1:} $C$ is simplicial.
    
    Let $v$ be a $k$-leaf in $C$, let $C'$ denote the clique adjacent to $C$, and let $T'=T-\{v\}$. Moreover, let $N,\overline{N},R,$ and $\overline{R}$ denote $N(T';C'),\overline{N}(T ;C'),R(T';C')$ and $\overline{R}(T';C')$, respectively. We have 

    $$\mu(T;C)=\frac{N+R+k}{N+1}$$
    and
    $$\mu(T)=\frac{2R+\overline{R}+N+k^2}{2N+\overline{N}+k}.$$

    We want to prove that
    $$(2N+\overline{N}+k)(2\mu(T)-\mu(T;C))>0,$$
    which is equivalent to
    $$2R+2\overline{R} + 2k^2 - 2k -\frac{(N+R+k)(\overline{N}+k-2)}{N+1} > 0.$$
    By the induction hypothesis, we have
    $$2\frac{R+\overline{R}}{N+\overline{N}} = 2\mu(T') > \mu(T';C') = \frac{R}{N},$$
    so it is sufficient to prove that
    $$\frac{R(N+\overline{N})}{N} + 2k^2 - 2k -\frac{(N+R+k)(\overline{N}+k-2)}{N+1} > 0.$$
    Multiplying by $\frac{N+1}{N}$, this is seen to be equivalent to
    $$R - \overline{N} + 2k^2-3k+2 - \frac{(k-3)R + k\overline{N} - k^2}{N} + \frac{\overline{N}R}{N^2} > 0.$$
    This can be broken up into three terms as follows:
    $$\frac{(N^2-kN+3N+\overline{N})(R-\overline{N})}{N^2} + \Big( \frac{\overline{N}}{N} - k \Big)^2 + \Big( (k-1)(k-2) + \frac{k^2}{N} + \frac{3\overline{N}}{N} \Big) > 0.$$
    Note here that the second term is trivially nonnegative, and the last term trivially positive. Since $|T'| \geq 2k$, we have $N > k$ (one gets at least $k+1$ sub-$k$-trees containing $C'$ by successively adding vertices); hence $N^2-kN+3N+\overline{N} > 0$. Thus the first term is positive by Lemma~\ref{lem:RgeqN}, completing the induction step in this case.
    
    \textbf{Case 2:} $C$ is not simplicial.
    
    By Theorem~\ref{thr:prob4}, we only need to consider the case where $C$ has degree $1.$
    Let $v$ be the unique common neighbor of $C$, and let $C_i$, $1\le i \le k$, be the other $k$-cliques in the $(k+1)$-clique spanned by $C \cup \{v\}$.

    Let $T_i$, $1\le i \le k$, be the sub-$k$-trees rooted at $C_i$ (pairwise disjoint except for the vertices of the cliques $C_{i}$).
    Let $N_{i} = N(T_{i};C_{i})$, $\overline N_{i} = \overline N(T_{i};C_{i})$, $R_{i} = R(T_{i};C_{i})$, $\overline R_{i} = \overline R(T_{i};C_{i})$, $\mu_{i} = \mu(T_{i};C_{i})$ and $\mu^\bullet_{i} = \mu^\bullet(T_{i};C_{i})$.
    We can assume without loss of generality that $N_1 \ge N_2 \ge \cdots \ge N_j> 1=N_{j+1}= \ldots=N_k,$ where $j \ge 2$ since $C$ is not simplicial.

    We can now express the local and global mean in a similar way to Case 1.
    Here $N(T; C)=\prod_{i=1}^k N_i +1$, and all the sub-$k$-trees counted here, except for $C$, contain $v$. We have
    $$\mu(T; C)= \frac{ \prod_{i=1}^k N_i (1+ \sum_{i=1}^k \mu_i^\bullet)}{\prod_{i=1}^k N_i +1} +k,$$
    $$\mu(T)= \frac{ \prod_{i=1}^k N_i (1+ \sum_{i=1}^k \mu_i^\bullet+k) +\sum_{i=1}^k (R_i +\overline R_i)  +k}{\prod_{i=1}^k N_i +\sum_{i=1}^k (N_i +\overline N_i)+1}.$$

In the remainder of this section, we will omit the bounds in products and sums if they are over the entire range from $1$ to $k$: $\sum N_i$ and $\prod N_i$ mean $\sum_{i=1}^k N_i$ and $\prod_{i=1}^k N_i$ respectively.

Then $\left( \prod N_i +\sum (N_i +\overline N_i)+1 \right) (2 \mu(T)-\mu(T; C))$ equals
\begin{equation}\label{eq:want-to-be-nonneg}
(\prod N_i)(1+ \sum \mu_i^\bullet+k) +k+2\sum (R_i +\overline R_i)  -k\sum (N_i +\overline N_i) - \frac{ (\prod N_i) (1+ \sum \mu_i^\bullet)\sum (N_i +\overline N_i)}{\prod N_i +1}.
\end{equation}
We want to show that this expression is positive.
By induction, we know that $2(R_i+\overline R_i) > (k+ \mu_i^\bullet)(N_i +\overline N_i)$, thus $2(R_i+\overline R_i) - k(N_i +\overline N_i) > \mu_i^\bullet(N_i +\overline N_i)$.
It follows that~\eqref{eq:want-to-be-nonneg} is greater than
$$
(\prod N_i) (1+ \sum \mu_i^\bullet+k) +k +\sum \mu_i^\bullet(N_i +\overline N_i) - \frac{ (\prod N_i) (1+ \sum \mu_i^\bullet)\sum (N_i +\overline N_i)}{\prod N_i +1}.
$$
Multiplying by $\frac{\prod N_i +1}{\prod N_i}$ and observing that this factor is greater than $1$, we find that~\eqref{eq:want-to-be-nonneg} is indeed positive if we can prove that
\begin{equation}\label{eq:crucialIneq2M-mu}
(\prod N_i+1) (1+ \sum \mu_i^\bullet+k) +k+\sum \mu_i^\bullet(N_i + \overline N_i) \ge (1+ \sum \mu_i^\bullet) \sum (N_i +\overline N_i).
\end{equation}

In particular, a potential counterexample would have to satisfy \begin{equation}\label{eq:obs}
    \frac{ (1+ \sum \mu_i^\bullet)\sum (N_i +\overline N_i)}{\prod N_i +1}>1+ \sum \mu_i^\bullet+k.
\end{equation}

\newpage
To simplify proving~\cref{eq:crucialIneq2M-mu}, we first note that it is sufficient to consider the case where $k=j.$
Once $N_i, \mu_i^\bullet$ and $ \overline N_i$ are fixed for $1 \le i \le j,$ the terms that are dependent on $k$ are 
$(\prod_{i \le j} N_i +1) k + k$ on the left,
and $(1+\sum_{i \le j} \mu_i^\bullet)(k-j)$ on the right. The latter since if $N_i=1$, then $T_i$ only consists of $C_i$, and thus $\overline N_i = \mu_i^\bullet = 0$.
Now since $\prod_{i \le j} N_i +1 \ge 1+ \sum_{i \le j} N_i > 1+\sum_{i \le j} \mu_i^\bullet,$ the increase of the left side is larger than the increase on the right side.
Here  $\prod_{i \le j} N_i \ge \sum_{i \le j} N_i$ is true since the product of $j \ge 2$ numbers, each greater than or equal to $2$, is at least equal the sum of the same $j$ numbers. Moreover, $N_i > \mu_i^\bullet$ follows from Lemma~\ref{lem:mu_ifv_N}.

So from now on, we can assume that $j = k$ and all $N_i$ are at least equal to $2.$

By \cref{lem:RgeqN}, $\overline N_i \le R_i=\mu_i N_i =(k+\mu_i^\bullet)N_i$.
Since~\eqref{eq:crucialIneq2M-mu} is a linear inequality in each $\overline N_i$ and the coefficient on the right-hand side is always greater than the coefficient on the left-hand side, we can reduce \cref{eq:crucialIneq2M-mu} to a sufficient inequality that is only dependent on $k$, $N_i$ and $\mu_i^\bullet$ for $1 \le i \le k$ by taking $\overline N_i =(k+\mu_i^\bullet)N_i$. This will be assumed in the following.

If now all the parameters in~\eqref{eq:crucialIneq2M-mu} are fixed except for one $\mu_i^\bullet$, we have a linear inequality in $\mu_i^\bullet$: the quadratic terms stemming from $\mu_i^{\bullet} \overline N_i$ are equal on both sides and cancel.
As such, it is sufficient to prove the inequality for the extremal values of $\mu_i^\bullet$. Here we use the trivial inequality $\mu_i^\bullet \ge 0$ as well as the upper bound $\mu_i^\bullet \le \frac{N_i-1}{2}$, which is taken from \cref{lem:mu_ifv_N}. So if~\eqref{eq:crucialIneq2M-mu} can be proven in the case where $\overline N_i =(k+\mu_i^\bullet)N_i$ and $\mu_i^\bullet \in \{0,\frac{N_i-1}{2}\}$ for all $i$ with $1 \le i \le k$, we are done.

For $2 \le k \le 5$, this is achieved by exhaustively checking all $2^k$ cases that result (using symmetry, there are actually only $k+1$ cases to consider). See the detailed verifications in \url{https://github.com/StijnCambie/AvSubOrder_ktree}.
So for the rest of the proof, we assume that $k \ge 6$, and we will use the slightly weaker bound $\mu_i^\bullet \le \frac{N_i}{2}$ instead of $\mu_i^\bullet \le \frac{N_i-1}{2}$ for $1 \le i \le k$ in a few cases. 

We distinguish two further cases, depending on the value of $\mu_1^\bullet$.
In these cases, we will use the following two inequalities.

\begin{claim}
    Let $k \ge 6$ and $N_1 \ge N_2 \ge \ldots \ge N_k \ge 2$.
    Then 
    \begin{align}
         \prod_{i=2}^k N_i &\ge 3  \sum_{2\le i \le k} N_i,\label{itm:a}\\
    \frac {5}{3} (\prod_{i=1}^k  N_i+1) &\ge  \left(  \sum_{2\le i \le k} N_i -2 \right)\sum_{1\le i \le k}N_i \label{itm:b} .
    \end{align}
\end{claim}

\begin{claimproof}
    The first inequality,~\cref{itm:a}, is true if all the $N_i$ are equal to $2$, since $2^{k-1}>6(k-1)$ for every $k \ge 6.$ Increasing some $N_i$ by $1$ increases the product by at least $2^{k-2}$, while the sum increases by only $3$. So the inequality holds by a straightforward inductive argument.

    Next, we prove~\cref{itm:b}.
    When all $N_i$ are equal to $2$, it becomes $\frac 53(2^k+1)\ge 4(k-2)k$. This is easily checked for $k \in\{6,7\}$, and for $k \ge 8$, the stronger inequality $2^k \ge 4k^2$ can be shown by induction.

    Now observe that the difference between the left- and right-hand sides is increasing with respect to $N_1$, since $\prod_{i \ge 2} N_i > \sum_{i \ge 2} N_i$ by the first inequality. It is also increasing in the other $N_i$'s; for example, we can see this is true for $N_2$
    since 
    $$\frac {5}{3} \prod_{i \not=2} N_i \ge 5 \sum_{i \not=2} N_i  > 2 \sum_{1\le i \le k}N_i > \left(  \sum_{2\le i \le k} N_i -2 \right) + \sum_{1\le i \le k}N_i  .$$
    In the first step, we have applied~\cref{itm:a} but replacing $N_1$ with $N_2$. Again, we may conclude using induction.
\end{claimproof}

\begin{claim}\label{clm:N_i/2everywhere}
    Given $\mu_1^\bullet= \frac{N_1}{2}$, it is sufficient to consider the case where $\mu_i^\bullet= \frac{N_i}{2}$ for all $1 \le i \le k.$
\end{claim}

\begin{claimproof}
    Starting from any counterexample to~\eqref{eq:crucialIneq2M-mu}, we can iteratively change $\mu_i^\bullet$ (considered as variables) for $1\le i\le k$ based on the worst case of the linearization (to $0$ if the coefficient on the left-hand side is greater and to $\frac{N_i}{2}$ if the coefficient on the right-hand side is greater) to obtain further counterexamples.
    
    To show that it suffices to consider $\mu_i^\bullet= \frac{N_i}{2}$ for every $1 \le i \le k$, we prove that, given $\mu_1^\bullet= \frac{N_1}{2}$, the coefficient of $\mu_2^\bullet$ on the left-hand side in~\cref{eq:crucialIneq2M-mu} is not greater than the coefficient on the right-hand side. This then implies that $\mu_2^\bullet= \frac{N_2}{2}$ is indeed the worst case. For $3 \le i \le k$, we can argue in the same fashion.
    
    Assume for sake of contradiction that the coefficient on the left-hand side is greater. Recall that $\overline{N}_2 = (k+\mu_2^\bullet)N_2.$
    After subtracting $\mu_2^\bullet(N_2+\overline N_2)$ from both sides, the coefficient of $\mu_2^\bullet$ on the left is $\prod N_i +1$, while on the right it is
    $\sum_{i \not=2} (N_i +\overline N_i) + (1+\sum_{i \not=2}  \mu_i^\bullet)N_2$. Thus we must have
$$\prod N_i +1 > \sum_{i \not=2} (N_i +\overline N_i) + (1+\sum_{i \not=2}  \mu_i^\bullet)N_2 \ge \sum_{i \not=2} (N_i +\overline N_i) + (1+\mu_1^\bullet)N_2.$$
    Using that
$$(1+ \mu_1^\bullet)N_2 = \left(1 + \frac{N_1}{2}\right)N_2 \ge \left(1 + \frac{N_2}{2}\right)N_2 \ge (1+ \mu_2^\bullet)N_2 = N_2 +\overline N_2-kN_2$$ 
(recall here that we are assuming without loss of generality that $N_1 \ge N_2 \ge \cdots \ge N_k$), this implies that
$\prod N_i +1 \ge \sum (N_i +\overline N_i)-kN_2.$
Adding $kN_2$ to both sides and multiplying both sides by $1+ \sum \mu_i^\bullet $
results in
$$\left( \prod N_i +1 + kN_2 \right)(1+ \sum \mu_i^\bullet) \ge (1+ \sum \mu_i^\bullet)  \sum (N_i +\overline N_i) >  (1+ \sum \mu_i^\bullet+k) (\prod N_i +1). $$
In the second inequality, we applied~\eqref{eq:obs} which we may do since we began with the assumption that we have a counterexample to~\eqref{eq:crucialIneq2M-mu}.
After simplification, we get that 
$N_2(1+ \sum \mu_i^\bullet) >\prod N_i$, and thus
$$\sum_{i \not=2} N_i> 1+\frac 12 \sum N_i\ge 1+\sum \mu_i^\bullet>\prod_{i\not=2} N_i.$$
Since $k \ge 6$, 
this is a clear contradiction to~\cref{itm:a}.
\end{claimproof}

Having proven Claim~\ref{clm:N_i/2everywhere}, we are left with two cases to consider: $\mu_i^\bullet= \frac{N_i}{2}$ for all $1 \le i \le k$, or $\mu_1^\bullet = 0$. It is easy to conclude in the former case, except when $k=6$ and at least $5$ values $N_i$ are equal to $2$, which has to be handled separately.
See \url{https://github.com/StijnCambie/AvSubOrder_ktree/blob/main/2M-mu_j_large_case1.mw} for details. 

The final remaining case is when $\mu_1^\bullet=0.$
We obtain two new inequalities by multiplying~\cref{itm:b} with $\frac{k+1}{2}$ and~\cref{itm:a} with $(1+ \sum \mu_i^\bullet)  \frac{N_1}6$, and use that $N_1 =\max\{ N_i\}$ and $\mu_i^\bullet \le \frac{N_i-1}{2}.$

\begin{align}
     \frac {5(k+1)}{6} (\prod N_i+1) &\ge \frac {k+1}2 \left(  \sum_{2\le i \le k} N_i -2 \right)\sum_{1\le i \le k} N_i\ge
    (1+ \sum \mu_i^\bullet) \sum_{1\le i \le k} (1+k)N_i,
     \label{itm:2} \\
     \frac{\prod N_i}6 (1+ \sum \mu_i^\bullet) &\ge (1+ \sum \mu_i^\bullet)  \frac{N_1}{2} \sum_{2\le i \le k} N_i \ge (1+ \sum \mu_i^\bullet) \sum_{2\le i \le k} N_i \mu_i^\bullet. \label{itm:3}
\end{align}
Summing these two inequalities together, we have that~\cref{eq:crucialIneq2M-mu} holds as a corollary of
$$(\prod N_i+1) (1+ \sum \mu_i^\bullet+k)\ge (1+ \sum \mu_i^\bullet) \sum \left(N_i(1+k+\mu_i^\bullet)\right).$$
\end{proof}

We remark that by considering a suitable $k$-broom, one can show that \cref{thr:prob2} is sharp, as was also the case for trees.

\section{Conclusion}\label{sec:conc}

This paper together with~\cite[Thm.~18 \& 20]{LX23} answers all of the open questions from \cite{SO18} except for one, which was stated as rather general and open-ended:
\begin{problem}\label{prob:5}
 For a given $r$-clique $R$, $1\leq r<k$, what is the (local) mean order of all sub-$k$-trees containing $R$?
\end{problem}

One natural version of this question is to consider the local mean sub-$k$-tree order over sub-$k$-trees that contain a fixed vertex. In this direction, we prove the following result, which can be considered as another monotonicity result related to~\cite[Thm.~23]{LX23}.

\begin{theorem}
    Let $T$ be a $k$-tree, $k \ge 2$, $C$ a $k$-clique of $T$, and $v$ a vertex in $C$. Then $\mu(T; v) < \mu(T; C).$
\end{theorem}
    
\begin{proof}
    The statement is trivially true if $\abs T \le k+1.$ So assume that $T$ with $\abs T \ge k+2$ is a minimum counterexample to the statement. Recalling the decomposition into trees $T_{i,j}$ used earlier, note that all sub-$k$-trees containing $v$ and not $C$ are part of $T_{i,j}$ for some $i, j$.
    Without loss of generality, we can assume that $\mu(T_{1,1}; v)=\max_{i,j} \mu(T_{i,j}; v).$
    It suffices to prove that $\mu(T_{1,1}; v)< \mu(T; C)$. Taking into account~\eqref{eq:mu_bullet(TC)}, it is sufficient to consider the case where $C$ is a simplicial $k$-clique of $T$. 
    Let $u$ be the simplicial vertex of $B_1$.
    Let $T'=T \setminus \{u\}$ and $C'=B_1 \setminus \{u\}$.
    \begin{claim}
        We have $\mu(T'; C')< \mu( T; C)$.
    \end{claim}
    \begin{claimproof}
        Let $R=R(T'; C')$ and $N=N(T'; C').$
        We now need to prove that $\mu( T; C)=\frac{2R+N+k}{2N+1}> \frac{R}{N}=\mu(T'; C')$, which is equivalent to $N+k>\frac{R}{N}.$ The latter is immediate since $N \ge \abs {T'} -(k-1)$ and $\frac RN=\mu(T'; C') \le \abs {T'}.$ 
    \end{claimproof}
    Since the sub-$k$-trees containing $v$ are exactly those that contain $C$, or sub-$k$-trees of $T'$ containing $C'$, or $k$-cliques within $B_1$ different from $C$, we conclude that
    $\mu(T; v) < \mu(T; C).$
\end{proof}

Luo and Xu~\cite[Ques.~35]{LX23} also asked if for a given order, a $k$-tree attaining the largest global mean sub-$k$-tree order is necessarily a caterpillar-type $k$-tree.
In contrast with the questions in~\cite{SO18}, this question is still open for trees.
We prove that the local version (proven in~\cite[Thm.~3]{CWW21}), which states that for fixed order, the maximum is attained by a broom,
is also true for the generalization of $k$-trees.
This is almost immediate by observations from~\cite{SO18}.

\begin{proposition}
 If a $k$-tree $T$ of order $n$ and $k$-clique $C$ of $T$ attain the maximum possible value of $\mu(T; C),$ then $T$ has to be a $k$-broom with $C$ being one of its simplicial $k$-cliques.
\end{proposition}

\begin{proof}
 Let $T'_C$ be the characteristic $1$-tree of $T$ with respect to $C.$
 Then by~\cite[Lem.~11]{SO18} $\mu(T; C) = \mu(T'_C; C)+k-1$.
 Since $T'_C$ is a tree on $n-(k-1)$ vertices, by~\cite[Thm.~3]{CWW21}, the maximum local mean subtree order is attained by a broom $B$. Since there is a $k$-broom $T$ with $C$ being a simplicial $k$-clique for which $T'_C\cong B$ with $C$ as root, this maximum can be attained.
 Reversely, if $\mu(T'_C; C)$ is maximized, then $T'_C$ is a broom where $C$ is its root (and thus simplicial).
\end{proof}

As such, we conclude that the $k$-tree variants of many results on the average subtree order for trees are now also proven.
The analogue of~\cite[Ques.~(7.5)]{jamison_average_1983} can be considered as the only question among them where the answer is slightly different for $k$-trees: in contrast with the case of trees ($k = 1$), the maximum local mean sub-$k$-tree order cannot occur in a $k$-clique with degree $2$ when $k \ge 2$ (with one small exception).

\section*{Acknowledgments}
The authors would like to express their gratitude to the American Mathematical Society for making a research visit possible related to the Mathematics Research Community workshop ``Trees in Many Contexts''.
This was supported by the National Science Foundation under Grant Number DMS $1916439$. 
The first author has been supported by internal Funds of KU Leuven (PDM fellowship PDMT1/22/005). The third author is supported by the
Swedish research council (VR), grant 2022-04030. The fourth author is supported by the National Institutes of Health (R01GM126554).

\paragraph{Open access statement.} For the purpose of open access,
a CC BY public copyright license is applied
to any Author Accepted Manuscript (AAM)
arising from this submission.

\bibliographystyle{abbrv}
\bibliography{ref}

\end{document}